\documentclass[11pt]{article}
\usepackage{amsmath,amssymb,latexsym,graphicx}
\usepackage{bm}
\usepackage{xcolor}

\newtheorem{theo}{Theorem}[section]
\newtheorem{propo}[theo]{Proposition}

\input amssym.def
\newsymbol\rtimes 226F
\newfont{\nset}{msbm10}
\newcommand{\ns}[1]{\mbox{\nset #1}}

\def\M{{\bm{M}}}

\def\S{\mbox{\boldmath $S$}}
\def\U{\mbox{\boldmath $U$}}
\def\V{\mbox{\boldmath $V$}}
\def\Z{\ns{Z}}

\def\b{\mbox{\boldmath $b$}}
\def\e{\mbox{\boldmath $e$}}

\def\colc{\mbox{\boldmath $c$}}

\def\a{\mbox{\boldmath $a$}}
\def\b{\mbox{\boldmath $b$}}
\def\e{\mbox{\boldmath $e$}}

\def\Cay{\mathop{\rm Cay }\nolimits}
\def\circ{\mathop{\rm circ }\nolimits}

\def\det{\mathop{\rm det }\nolimits}

\def\diag{\mathop{\rm diag }\nolimits}

\def\mod{\mathop{\rm mod }\nolimits}

\def\vec0{\mbox{\boldmath $0$}}
\def\vecx{\mbox{\boldmath $x$}}
\def\vecy{\mbox{\boldmath $y$}}

\begin{document}

\title{Abelian Cayley digraphs with asymptotically \\ large order
for any given degree
\thanks{
Research supported by the ``Ministerio de Educaci\'on y Ciencia"
(Spain) with the European Regional Development Fund under projects
MTM2011-28800-C02-01
and by the Catalan Research Council under project 2014SGR1147.
\newline \indent
 Emails: {matfag@ma4.upc.edu, fiol@mat.upc.es, sonia@ma4.upc.edu}
            }
}
\author{F. Aguil\'o, M.A. Fiol and S. P\'erez\\
\\ {\small Departament de Matem\`atica Aplicada IV}
\\ {\small Universitat Polit\`ecnica de Catalunya}
\\ {\small Jordi Girona 1-3 , M\`odul C3, Campus Nord }
\\ {\small 08034 Barcelona, Catalonia (Spain). %; email: {\tt fiol@mat.upc.es}
}   }

%\date{}
\maketitle
\begin{abstract}
Abelian Cayley digraphs can be constructed by using a generalization to $\Z^n$ of the concept of congruence in $\Z$.
Here we use this approach to present a family of such digraphs, which, for every fixed value of the degree, have asymptotically large number of vertices as the diameter increases. Up to now, the best known asymptotically dense results were all non-constructive.
\end{abstract}

\noindent
{\bf Keywords.}
Cayley digraph, Abelian group, Degree/diameter problem, Congruences in $\Z^n$, Smith normal form.

\noindent
{\bf AMS subject classifications.} 05012, 05C25.

\section{Introduction}

The degree-diameter problem for graphs (directed or undirected) has been widely studied in the last decades because of its relevance to the design of some interconnection or communication networks for parallel processors.
This is because one wants to have a large number of processors without requiring a large 
number of links, and  without incurring long delays in communication from one 
processor to another.
%See \cite{BeHeBaAr91} and for some examples.

In terms of directed graphs, the degree-diameter problem consists of, given a degree $d$ and a diameter $k$,  maximizing
the number of vertices of a digraph with maximum degree $d$ and diameter $k$. We refer to the survey Miller and Sir\'an \cite{ms13} for the current state of the art, and to the
the survey of Bermond, Comellas, and Hsu \cite{BeCoHs95} for more details about the history of the problem.

A desirable extra property of such dense digraphs is to be vertex-transitive (in particular regular, with in-degree and out-degree of every vertex equal to $d$). Then, the network is seen identically from any processor and one can implement easily all protocols, being basically the same for each node.  In fact, most of the literature so far has focused on a very important class of vertex-transitive digraphs, the Cayley digraphs  (see e.g. \cite{WoCo74,HsJi94,FiYeAlVa87}).
Let us recall that a {\em Cayley digraph} of a group $\Gamma$ with respect to a generating set $A$, denoted by
$G=\Cay(\Gamma ;A)$, is the digraph whose vertices are labeled with the
elements of $\Gamma$, and there is an arc $(u,v)$ if and only if $v-u \in A$.

In this paper we deal exclusively with the degree-diameter problem when the group $\Gamma$ of the Cayley digraph is an Abelian group.

Let $N\!A_{d,k}$ (respectively, $N\!C_{d,k}$) be the maximum number of vertices that a Cayley digraph of an Abelian group (respectively, of a cyclic group), with  degree $d$ and diameter $k$, can have.
In this framework, Wong and Coppersmith \cite{WoCo74} proved that, for fixed degree $d$ and large diameter $k$,
\begin{equation}
\label{boundsWC}
\left(\frac{k}{d} \right)^d+O(k^{d-1})\le N\!C_{d,k} \le \frac{k^d}{d!}+O(k^{d-1}).
\end{equation}
The exact value of $N\!C_{d,k}$ is only known
for the case of degree $d=2$, where the authors of \cite{MoFiFa85,FiYeAlVa87} proved that, for any $k\ge 2$,
\begin{equation}
 N\!C_{2,k}=\left\lceil \frac{(k+2)^2}{3}\right\rceil -1.\label{eq:NCmaxd2}
\end{equation}

The \textit{density}, $\delta(G)$, of a Cayley digraph $G$ on a $d$--generated finite Abelian group of order $N$ is defined by
\begin{equation}
\delta(G)=\frac{N}{(k(G)+d)^d}.\label{eq:dens}
\end{equation}
For $d=2$, from (\ref{eq:NCmaxd2}) and (\ref{eq:NAmaxd2}), we have $\delta\leq\frac13+\frac1{(k+2)^2}$. For $d=3$, Fiduccia, Forcade and  Zito \cite[Corollary~3.6]{FiFoZi98} gave the upper bound $N\!A_{3,k}\leq\frac3{25}(k+3)^3$ and so, for this degree, it follows that $\delta\leq\frac3{25}=0.12$. Several authors gave different families of dense Cayley digraphs of cyclic groups, also known in the literature as loop networks. Table~\ref{tab:d3} summarizes some of the results for the well-studied case $d=3$. The entry of $\delta^*$ has to be understood in the asymptotical sense for large values of $k$.

\begin{table}[h]
\centering
\begin{tabular}{|l|l|l|}\hline
Authors&$\delta^*$&Conditions on $k$\\\hline\hline
G\'omez, Guti\'errez \& Ibeas \cite{GoGuIb07}&$0.0370$&$k\equiv0,1,2\pmod{3}$\\
%Hsu \& Jia \cite{HsJi94}&$0.0620$&\\
Aguil\'o, Fiol \& Garc\'{\i}a \cite{AgFiGa97}&$0.0740$&$k\equiv2,4,5\pmod{6}$\\
Chen \& Gu \cite{ChGu92}&$0.0780$& No condition on $k$ \\
Aguil\'o \cite{Ag99}&$0.0807$&$k=22t+12$, $t\not\equiv2,7\pmod{10}$\\
Aguil\'o, Sim\'o \& Zaragoz\'a \cite{AgSiZa01}&$0.0840$&$k\equiv2\pmod{30}$\\\hline
\end{tabular}
%\vspace*{0.01\linewidth}
\caption{Several cyclic constructions for $d=3$}
\label{tab:d3}
\end{table}

Notice that the fact $N=\alpha k^d+O(k^{d-1})$ does not necessarily imply $\delta=\alpha$. For instance, a result of Dougherty and Faber \cite[Corollary~8.2]{DoFa04} gives the existence of cyclic Cayley digraphs of degree $d=3$ and order $N=0.084k^3+O(k^2)$ for all $k$. Table~\ref{tab:figures} gives numerical evidence that this result is not difficult to achieve; however, equality $\delta=0.084$ is not. As far as we know, for the cyclic case, there is only one known value of $4\leq N\leq8000$ that achieves this density, that is $N=84$ with (for instance) $G=\Cay(\Z_{84},\{2,9,25\})$ and $k(G)=7$. Moreover, no known order achieves the upper bound of Fiduccia, Forcade and Zito for the cyclic case. As far as we know, no explicit proposal is known achieving $\delta^*=0.084$, in the sense of Table~\ref{tab:d3}, for all $k$.

\begin{table}[h]
\centering
\begin{tabular}{|l|l|r|r|r|r|}\hline
$k$&$\delta$&$\lceil0.084k^3\rceil$&$\lceil\frac3{25}k^3\rceil$&$\lfloor\frac3{25}(k+3)^3\rfloor$& $N\!C_{3,k}$\\\hline\hline
$1$ & $0.06250$ & $1$ & $1$ & $7$ & $4$ \\ \hline
$2$ & $0.07200$ & $1$ & $1$ & $13$ & $9$ \\ \hline
$3$ & $0.07407$ & $3$ & $4$ & $24$ & $16$ \\ \hline
$4$ & $0.07872$ & $6$ & $8$ & $38$ & $27$ \\ \hline
$5$ & $0.07812$ & $11$ & $15$ & $56$ & $40$ \\ \hline
$6$ & $0.07819$ & $19$ & $26$ & $81$ & $57$ \\ \hline
$7$ & $0.08400$ & $29$ & $42$ & $111$ & $84$ \\ \hline
$8$ & $0.08340$ & $44$ & $62$ & $147$ & $111$ \\ \hline
$9$ & $0.07986$ & $62$ & $88$ & $192$ & $138$ \\ \hline
$10$ & $0.08011$ & $84$ & $120$ & $244$ & $176$ \\ \hline
\end{tabular}
%\vspace*{0.01\linewidth}
\caption{Several density values for $d=3$ and $1\leq k\leq10$ in the cyclic case}
\label{tab:figures}
\end{table}

It is worth to mention the work of R\"odseth~\cite{Rod96} on weighted loop networks that gave sharp lower bounds for the diameter and mean distance for $d=2$ and general bounds for $d\geq3$ as well.

As mentioned above, the concern of our paper are Cayley digraphs of Abelian groups.
In this case, the bounds in \eqref{boundsWC} also apply for $N\!A_{d,k}$.
In  \cite[Th. 1.1]{MaSchJi11} Mask, Schneider, and Jia claimed to show that, for any $d$ and $k$,
$N\!A_{d,k}=N\!C_{d,k}$, but Fiol in \cite{Fi14} corrected this claim by giving some counterexamples for the case $d=2$. In fact, we have that for such a degree
Fiol et al. in \cite{MoFiFa85,FiYeAlVa87}:
\begin{propo}
For any diameter $k\ge 2$,
\begin{equation}
N\!A_{2,k}=\left\{
\begin{array}{ll}
N\!C_{2,k}+1, & \mbox{if $k\equiv 1$\ $(\mod 3)$}, \\
N\!C_{2,k}, & \mbox{otherwise.}
\end{array}
\right.\label{eq:NAmaxd2}
\end{equation}
\end{propo}
(see also the comments of Dougherty and Faber in \cite{DoFa04}).

%An example and its hiper-L:
%Let $\Gamma=\Z_{111}$ and $A=\{1,31,69\}$.
%\begin{figure}
%%\begin{center}
%%\vskip -.3cm
%%\hskip -8cm
%\includegraphics[width=16cm]{IEEE-TC}
%%\end{center}
%\end{figure}

In general, if $\Gamma$ is an Abelian group and $|A|=d$, the Cayley digraph $\Cay(\Gamma,A)$ has order:
$$
N\!A_{d,k}< {k+d\choose d}={k+d\choose k}.
$$
Then, it is not difficult to see that:
$$
k> \sqrt[d]{d!N\!A_{d,k}}-\frac{1}{2}(d+1).
$$

Nevertheless, as far as we know, the best upper bound known for the order of Cayley digraphs of Abelian groups is the one of
Dougherty and Faber
in \cite[Theorem~21]{DoFa04}, where they proved the following non-constructive result:\newline
{\it There is a positive constant $c$ (not depending on $d$ or $k$), such that, for any fixed $d\ge 2$ and any $k$, there exist Cayley digraphs of Abelian groups on $d$ generators having diameter at most $k$ and number of vertices $N_{d,k}$ satisfying}
$$
N\!A_{d,k}\ge \frac{c}{d(\ln d)^{1+\log_2 e}}\frac{k^d}{d!}+O(k^{d-1}).
$$

In our study, we use the following approach developed by Aguil\'o, Esqu\'e and Fiol \cite{EsAgFi93,Fi95}: Every Cayley digraph $G$ from an Abelian group $\Gamma$ is fully characterized by an integral
$n\times n$ matrix $\M$ such that $\Gamma=\Z^n/\M\Z^n$ (the so-called `group of integral $n$-vectors modulo $\M$' that is detailed in Section~\ref{sec:thbk}). Then, in such a representation, the digraph $G$ is isomorphic to $\Cay(\Z^n/\M\Z^n,A)$, where $A$ is the set of unitary coordinate vectors $\e_i$, $i=1,\ldots,n$.

The plan of the paper is as follows. In the following section and for the sake of completeness, we will recall the algebraic background on which the above isomorphism is based. Afterwards, in contrast with the theoretical bound of Dougherty and Faber, we will present an explicit infinite family of Cayley digraphs of Abelian groups whose order is asymptotically large. This family will be constructed using a generalization of the concept of congruence in $\Z$ to $\Z^n$.

\section{Some theoretical background}\label{sec:thbk}

In this section we recall some basic concepts and results on which our study
is based (see \cite{EsAgFi93,Fi87,Fi95} for further details).

\subsection{Congruences in $\Z^{n}$}
Given an nonsingular integral $n\times n$ matrix $\M$, we say that
the integral vectors $\a,\b\in\Z^n$ are {\em congruent modulo $\M$} (see \cite{Fi87}) when its difference belongs to the lattice generated by $\M$, that is,
\begin{equation}
\label{cong-nDIM}
\a\equiv \b \quad (\mod \M)\qquad \iff \qquad \a-\b\in \M\Z^n.
\end{equation}
So,  as the quotient group $\Z_m=\Z/m\Z$ is the cyclic group of integers modulo $m$,
the quotient group $\Z^n_{\M}=\Z^{n}/\M\Z^{n}$ can intuitively be called the
{\it group of integral vectors modulo $\M$} (where each equivalence class is identified by any of its representatives).

In particular, notice that,
when $\M=\diag (m_1,m_2,\ldots,m_n)$, \eqref{cong-nDIM} implies that the vectors $\a=(a_1,a_2,\ldots,a_n)^\top$ and
$\b=(b_1,b_2,\ldots,b_n)^\top$ are congruent modulo $\M$ if and only if
$$
a_i\equiv b_i\quad (\mod m_i)\qquad (1\le i\le n).
$$
Moreover, in this case $\Z^n_{\M}$ is the direct product of the cyclic groups
$\Z_{m_i}$, $i=1,2,\ldots,n$.

\subsection{The Smith normal form}
As before, let $\M=(m_{ij})$ be a nonsingular matrix of $\Z^{n\times n}$, with $N=\det \M\neq 0$.
Let $k\in \Z$, $1\leq k\leq n$. The $k$th {\it determinantal divisor} of
$\M$, denoted by $d_{k}(\M)=d_{k}$, is defined as the greatest common
divisor of the $(^{n}_{k})^{2}$  $k\times k$ determinantal minors of $\M$.
%Since $\M$ is nonsingular, not all of them are zero.
Notice that $d_{k}\mid d_{k+1}$ for all $k=1,2,\ldots ,n-1$ and $d_{n}=m$. For convenience,
put $d_{0}=1$.
The {\it invariant factors} of $\M$ are the quantities
$$
s_{k}(\M) = s_{k} = \frac{d_{k}}{d_{k-1}},\qquad k=1,2,\ldots ,n.
$$
It can be shown that $s_{i}\mid s_{i+1}$, $i=1,2,\ldots ,n-1$.

%%%%%%%%%%%%%%%%%%%%%%%%%%%%%%%%%%%%%%%%%%%%%%%%%%%%%%%%%%%%%%%%%%%%
By the {\em Smith normal form} theorem, $\M$ is equivalent to the diagonal
matrix $\S(\M)=\S=\;$diag$(s_{1},s_{2},\ldots ,s_{n})$, i.e. there are two unimodular matrices $\U,\V\in\Z^{n\times n}$ such that $\S=\U\M\V$. This canonical form $\S$ is unique and the unimodular matrices $\U$ and $\V$ are not. For more details, see e.g. Newman \cite{Ne72}.

\begin{propo} {\rm (Fiol~\cite{Fi95})}
\label{pro2.1}
Set $\M\in\Z^{n\times n}$ with $N=|\det\M|$.
\begin{itemize}
\item[$(a)$]
The number of equivalence classes modulo $\M$ is
 $|\Z^{n}/\M\Z^{n}| = N$.
\item[$(b)$]
 If $p_{1}^{r_{1}}p_{2}^{r_{2}} \cdots p_{t}^{r_{t}}$ is the prime
factorization of $N$, then $\Z^{n}/\M\Z^{n} \cong \Z^{r}/\S'\Z^{r}$ for some
$r\times r$ matrix $\S'$ with $r \leq \max \{r_{i} :1 \leq i \leq t \}$.
\item[$(c)$]
 The $($Abelian$)$ group of integral vectors modulo $\M$ is cyclic
if and only if $d_{n-1}=$1.
\item[$(d)$]
 Let $r$ be the smallest integer such that $s_{n-r}=1$. Then $r$ is
the rank of $\Z^{n}/\M\Z^{n}$ and the last $r$ columns of $\U^{-1}$ form a basis
of $\Z^{n}/\M\Z^{n}$.\ $\Box$
\end{itemize}
\end{propo}

In other words, when $s_1=\cdots=s_{n-k-1}=1$ and $s_{n-k}>1$, we have $\S'=\diag(s_{n-k},\ldots,s_n)$ and $\Z^{n}/\M\Z^{n} \cong \Z^{r}/\S'\Z^{r}\cong\Z_{s_{n-r}}\oplus\cdots\oplus\Z_{s_n}$. The isomorphism is given by $\phi(\vecx)=\U\vecx$ and it will be used in the next section.

\section{A new family of Abelian Cayley digraphs}
Let us consider the circulant matrix $\M=\circ(n,-1,-1,\ldots,-1)$, which defines the digraph of commutative steps $G_{\M}=\Cay (\Z^n/\M\Z^n, \{e_1,e_2,\ldots,e_n\})$.
Clearly, $G_{\M}$ is regular with degree $d=n$. Moreover, its order is
\begin{align*}
N=\det \M & =
\det
\left(
\begin{array}{ccccc}
n & -1 & -1 & \cdots & -1 \\
-1 & n & -1 & \cdots & -1 \\
-1 & -1 & n & \cdots & -1 \\
\vdots & \vdots & \vdots & & \vdots \\
-1 & -1 & -1 & \cdots & n
\end{array}
\right)
=\det
\left(
\begin{array}{ccccc}
1 & -1 & -1 & \cdots & -1 \\
1 & n & -1 & \cdots & -1 \\
1 & -1 & n & \cdots & -1 \\
\vdots & \vdots & \vdots & & \vdots \\
1 & -1 & -1 & \cdots & n
\end{array}
\right) \\
 &=\det
\left(
\begin{array}{ccccc}
1 & 0 & 0 & \cdots & 0 \\
1 & n+1 & 0 & \cdots & 0 \\
1 & 0 & n+1 & \cdots & 0 \\
\vdots & \vdots & \vdots & & \vdots \\
1 & 0 & 0 & \cdots & n+1
\end{array}
\right)=(n+1)^{n-1}.
\end{align*}

%\begin{figure}[htb]
%\begin{center}
%\includegraphics*[scale=.75]{proyec-H02}
%\end{center}
%\caption{Obtaining the $q$'s and the $p$'s by projecting $H_0$.}
%\protect\label{fig1}
%\end{figure}

Given $\vecx=(x_1,\ldots,x_n)$, let us denote $\| \vecx \|_1=\sum_{i=1}^n|x_i|$. The following result gives the diameter of $G_{\M}$.

\begin{propo}
Given the circulant matrix $\M=\circ(n,-1,-1,\ldots,-1)$, the digraph of commutative steps $G_{\M}$ defined as above has diameter $k={n\choose 2}=n(n-1)/2$.
\end{propo}

\begin{proof}
Let $H\subset \Z^n$ be a set of $N$ nonnegative integral vectors, which are different modulo $\M$.
Let $k(H)=\max\{\|\vecx\|_1=\sum_{i}x_i :~\vecx\in H\}$. Then, the diameter of $G_{\M}$ is $k_{G_{\M}}=\min\{k(H):\;H\subset\Z^n\}$. Let us assume that $L$ is a set that attains such a minimum, that is, $k=k(L)$. Thus, $L$ corresponds to an optimal set of lattice points, in the sense that the distance from the origin to a lattice point (vector)
 equals the distance from vertex zero to the corresponding vertex of $G_{\M}$. Then, let us see that  $\vecx\in L$ if and only if $0\le x_i\le n-1$ for $i=1,\ldots,n$, and
\begin{itemize}
\item[$(1)$]
There is at most 1 entry  such that $x_i\ge n-1$,
\item[$(2)$]
There are at most 2 entries  such that $x_i\ge n-2$,
\item[$\vdots$]
\item[$(n-1)$]
There are at most $n-1$ entries  such that $x_i\ge 1$.
\end{itemize}
Indeed, $\vecx\in L$ cannot have any entry, say $x_1\ge n$, since otherwise the vector
$$
\vecy=\vecx-(n,-1,-1,\ldots,-1)\equiv \vecx\ (\mod \M)
$$
would satisfy $\sum y_i =\sum x_i -1$. Moreover,
\begin{itemize}
\item[$(1)$]
There cannot be 2 entries, say $x_1,x_2\ge n-1$, since otherwise the vector
$$
\vecy=\vecx-(n,-1,-1,\ldots,-1)-(-1,n,-1,\ldots,-1)\equiv \vecx\ (\mod \M)
$$
would satisfy $\sum y_i =\sum x_i -2$,
\item[$(2)$]
There cannot be 3 entries, say $x_1,x_2,x_3\ge n-2$, since otherwise the vector
$$
\vecy=\vecx-(n,-1,-1,\ldots,-1)-(-1,n,-1,\ldots,-1)-(-1,-1,n,\ldots,-1)\equiv \vecx\ (\mod \M)
$$
would satisfy $\sum y_i =\sum x_i -3$,
%\item[$\vdots$]
\end{itemize}
and so on.
Now we will show that, under the above conditions, the cardinality of $L$ is $N=\det \M=(n+1)^{n-1}$.
With this aim, for every pair  of integers $0\le m\le n$, let $f(m,n)$ be the number of integral vectors with $m$ entries, $(x_1,\ldots,x_m)\in \Z^m$, $0\le x_i\le n-1$, having at most $i$ entries $x_j\ge n-i$ for every $1\le i\le m$. (Notice that if $m=n$, the case $i=n$ does not imply any restriction.)
Such a number satisfy the following recurrences (by definition, we take $f(0,n)=1$).

If $m<n$,
\begin{equation}
\label{eq1}
f(m,n)=\sum_{i=0}^m {m\choose m-i} f(i,n-1).
\end{equation}

If $m=n$,
\begin{equation}
\label{eq2}
f(m,n)=\sum_{i=0}^{m-1} {m\choose m-i} f(i,n-1).
\end{equation}
Now, by induction, we can prove that
\begin{equation}
\label{eq3}
f(n,m)=(n-m+1)(n+1)^{m-1}.
\end{equation}
\begin{itemize}
\item
$f(0,n)=1$ (by definition).
\item
$f(1,n)=n$ (obvious).
\item
Let assume that \eqref{eq3} holds for $m-1$. Then,

If $m<n$,
\begin{align*}
\label{eq1}
f(m,n) & =\sum_{i=0}^m {m\choose i} f(i,n-1)=\sum_{i=0}^m {m\choose i} (n-i)n^{i-1}
=\sum_{i=0}^m {m\choose i} n^{i}-\sum_{i=0}^m {m\choose i} in^{i-1}\\
& =(m+1)^m-m(n+1)^{m-1}=(n+1-m)(n+1)^{m-1}.
\end{align*}

If $m=n$,
\begin{equation*}
f(m,n) =\sum_{i=0}^{n-1} {n\choose n-i} f(i,n-1)=\sum_{i=0}^{n-1} {n\choose n-i} (n-i)n^{i-1}
=\sum_{i=0}^{n-1} {n-1\choose i} n^{i}=(n+1)^{n-1}.
\end{equation*}
\end{itemize}
Consequently, from \eqref{eq3}, $|L|=f(n,n)=(n+1)^{n-1}=\det \M$.

Finally, notice that, according to the characterization above, $k(L)$ equals the distance from the origin to the vectors of the form $\vecx=(n-1,n-2,\ldots,0)$ (up to permutation of the entries). Then,
$$
k_{G_{\M}}=k(L)=\sum_{i=1}^n x_{i} =1+2+\cdots+(n-1)=\frac{n(n-1)}{2}={n\choose 2},
$$
as claimed.
\end{proof}

\begin{figure}[h]
\centering
\includegraphics[width=0.15\linewidth]{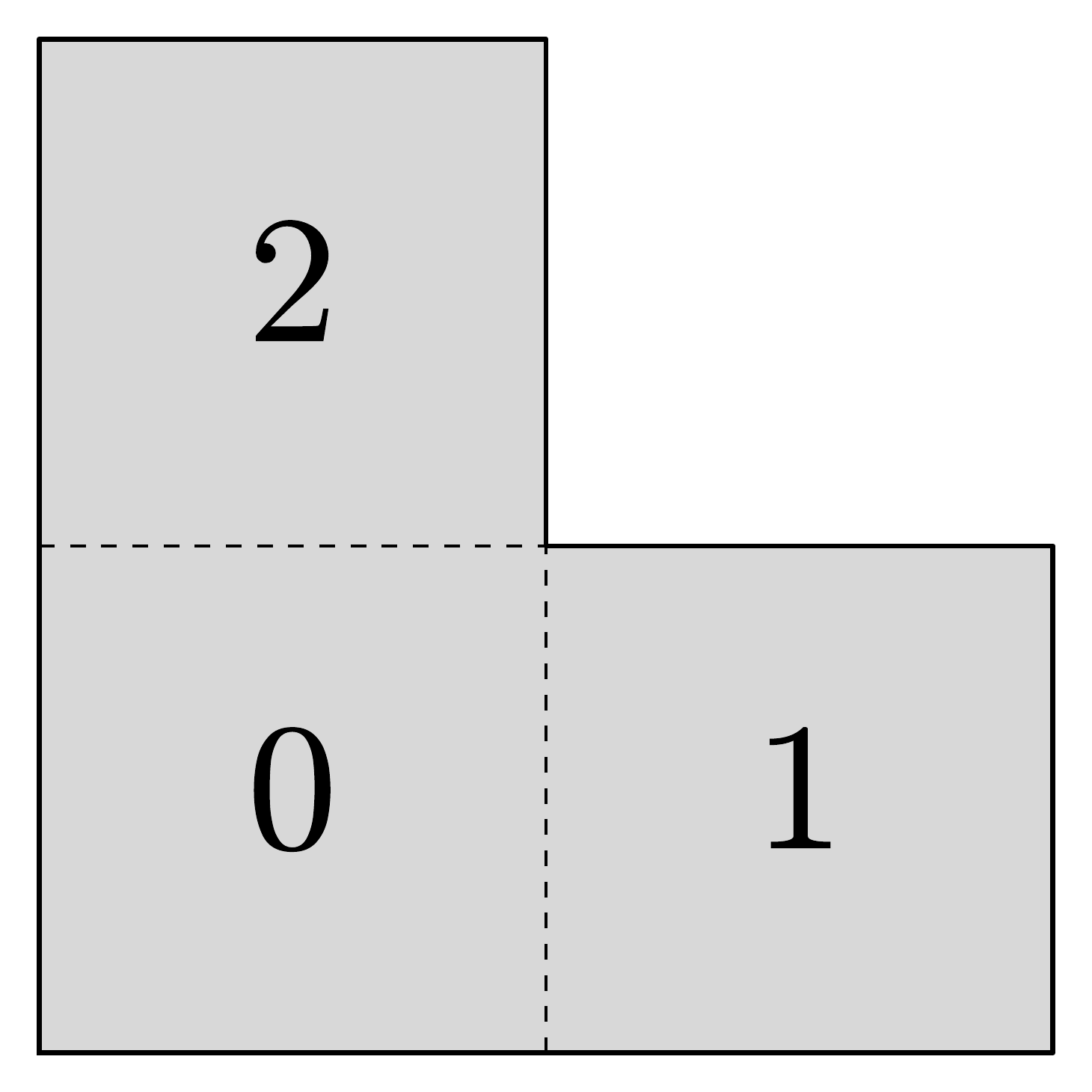}
\hspace{0.05\linewidth}
\includegraphics[width=0.35\linewidth]{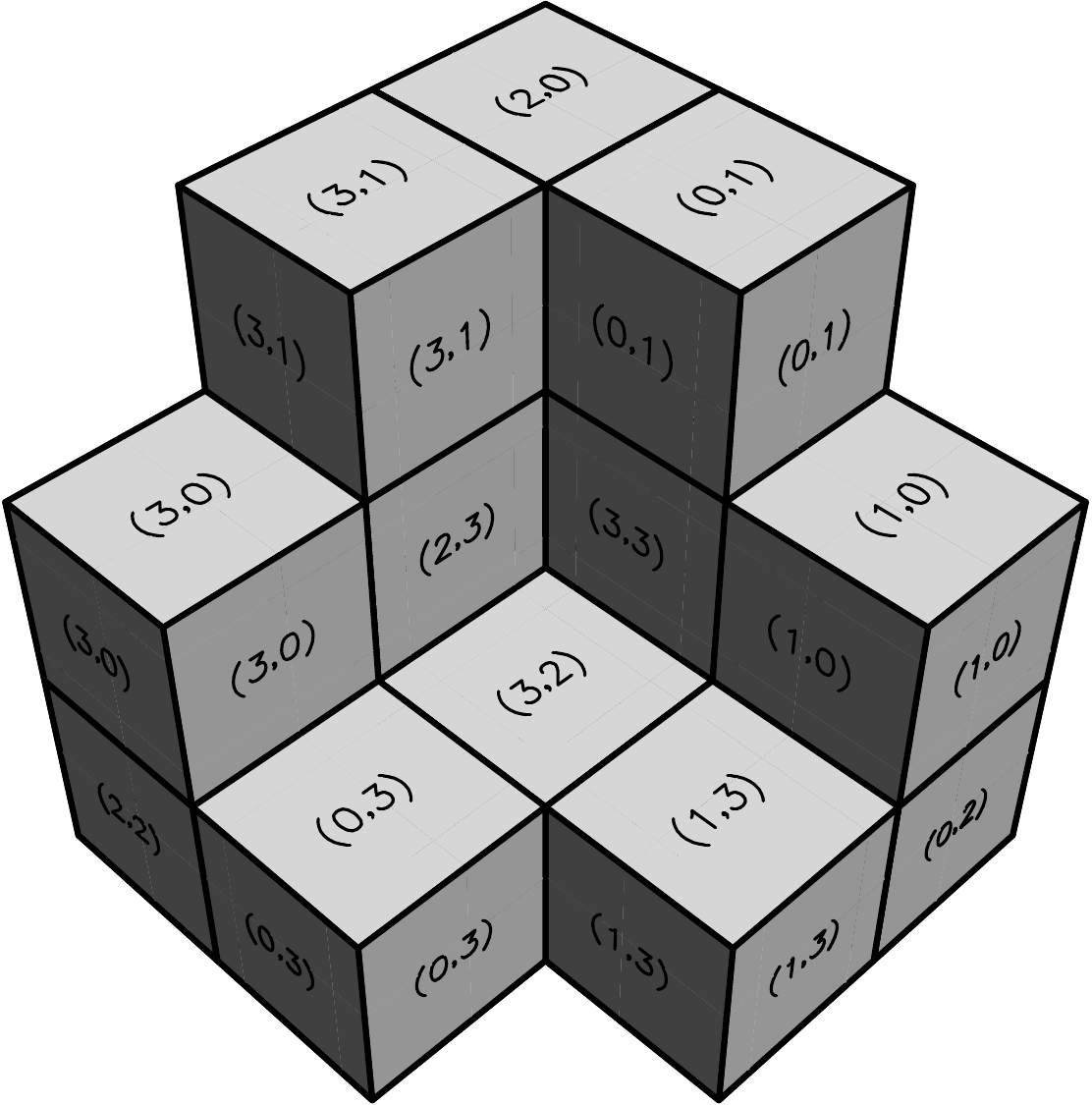}
\caption{Minimum distance diagrams related to $D_2$ and $D_3$}
\label{fig:exemple3d}
\end{figure}

Notice that, from the above proof, the number of vertices in $G_{\M}$ at maximum distance $k={n\choose 2}$ from every vertex is $n!$.

From now on we use the notation $s\Z_t=\Z_t\oplus\stackrel{(s)}{\cdots}\oplus\Z_t$.

\begin{theo}\label{teo:particular}
Set $B_n=\{(1,1,\ldots,1),(2,1,\ldots,1),\ldots,(1,\ldots,1,2)\}\subset\Z^{n-1}$. Then, the Cayley digraph $D_n=\Cay((n-1)\Z_{n+1};B_n)$ has diameter $k_n={n\choose 2}$.
\end{theo}
\begin{proof}
Taking the matrix $\M_n=\circ(n,-1,\ldots,-1)\in\Z^{n\times n}$, it is not difficult to see it has the Smith Normal Form $\S_n=\U_n\M_n\V_n=\diag(1,n+1,\ldots,n+1)$ with unimodular matrices
\begin{equation}
\U_n=\left(
\begin{array}{ccccc}
1 & 1 & 1 & \cdots & 1 \\
1 & 2 & 1 & \cdots & 1 \\
1 & 1 & 2 & \cdots & 1 \\
\vdots & \vdots & \vdots & & \vdots \\
1 & 1 & 1 & \cdots & 2
\end{array}
\right)\qquad\text{and}\qquad
\V_n=\left(
\begin{array}{crrcr}
1 & -1 & -1 & \cdots & -1 \\
0 & 1 & 0 & \cdots & 0 \\
0 & 0 & 1 & \cdots & 0 \\
\vdots & \vdots & \vdots & & \vdots \\
0 & 0 & 0 & \cdots & 1
\end{array}
\right).\label{eq:UiV}
\end{equation}
Thus, we have the isomorphism of digraphs $G_{\M_n}\cong D'_n=\Cay(\Z_1\oplus(n-1)\Z_{n+1};B'_n)$, with $B'_n=\{(1,1,1,\ldots,1),(1,2,1,\ldots,1),\ldots,(1,1,1,\ldots,1,2)\}\subset\Z^n$, given by $\phi(\e_i)=\U_n\e_i=\b_i$ ($\b_i$ are the elements of $B'_n$). Therefore, by Proposition~\ref{pro2.1}, the diameter of $D'_n$ is $k_n={n\choose 2}$. Now the statement follows from the direct digraph isomorphism $D'_n\cong D_n$.
\end{proof}

Figure~\ref{fig:exemple3d} shows minimum distance diagrams related to $D_2=\Cay(\Z_3\{1,2\})$ and $D_n=\Cay(\Z_4\oplus\Z_4,\{(1,1),(2,1),(1,2)\})$ with diameters $k_{D_2}=1$ and $k_{D_3}=3$.

\begin{propo}\label{pro:general}
Let us denote $\M_n=\circ(n,-1,\ldots,-1)$, $n\geq2$, the matrix of Proposition~\ref{pro2.1}. Consider the matrix $\M_{n,m}=m\M_n=\circ(mn,-m,\ldots,-m)\in\Z^{n\times n}$, for $m\geq1$. Then, the commutative step digraph $G_{\M_{n,m}}$ has order $N_{n,m}=m^n(n+1)^{n-1}$ and diameter $k_{n,m}={n+1\choose 2}m-n$.
\end{propo}
\begin{proof}
Let us denote the columns $\M_{n,m}=(\colc_1|\cdots|\colc_n)$, that is the $i$-th column of $\M_{n,m}$ is denoted by $\colc_i=(-m,\ldots,-m,\overbrace{mn}^{i},-m,\ldots,-m)^\top$. Let us assume that $L$ is a hyper-L for the digraph $G_{\M_{n,m}}$, in the same sense as in the proof of Proposition~\ref{pro2.1}. Set $\vecx=(x_1,\ldots,x_n)\in L$. Then, all the entries of $\vecx$ must be $x_i\leq mn-1$. If say $x_1\geq mn$, then $\vecy=\vecx-\colc_1=(x_1-mn,x_2+m,\ldots,x_n+m)\in\Z^n_{\geq0}$. Then, $\vecy\equiv\vecx\pmod{\M_{n,m}}$ and $\|\vecy\|_1=\|\vecx\|_1-m$, thus $\vecx\notin L$; a contradiction.

There cannot be more than $k$ entries of $\vecx$ with $x_i\geq m(n-k+1)-1$, for $k\in\{1,\ldots,n-1\}$. Otherwise, if it were the case that for instance $x_1,\ldots,x_{k+1}\geq m(n-k+1)-1$, then the vector
\[
\vecy_k=\vecx-\colc_1^\top-\ldots-\colc_{k+1}^\top=(x_1-mn+km,\ldots,x_{k+1}-mn+km,x_{k+2}+km,\ldots,x_n+km)
\]
would be $\vecy_k\equiv\vecx\pmod{\M_{n,m}}$, $\vecy_k\in\Z_{\geq0}^n$ and $\|\vecy_k\|_1=\|\vecx\|_1-m[k(n-k)+1]$. Thus $\vecx\notin L$, a contradiction.

Therefore, the maximum $\max\{\|\vecx\|_1:~\vecx\in L\}$ is attained at vectors of type $\vecx^*=(mn-1,m(n-1)-1,\ldots,m2-1,m-1)$ (up to permutations of the entries). So, the diameter is
$$
k_{n,m}=k(L)=\|\vecx^*\|_1={n+1\choose 2}m-n=n\left(\frac{n+1}2m-1\right).
$$
Using similar arguments as in the proof of Proposition~\ref{pro2.1}, it can be seen that there are $N_{n,m}=\det \M_{n,m}=m^n(n+1)^{n-1}$ different vectors in $L$, the order of the commutative step digraph $G_{\M_{n,m}}$.
\end{proof}

\begin{theo}
Consider the generator set $B'_n\subset\Z^n$ given in the proof of Theorem~\ref{teo:particular}. Then, the Cayley digraph $D_{n,m}=\Cay(\Z_m\oplus(n-1)\Z_{m(n+1)};B'_n)$ has diameter $k_{n,m}={n+1\choose 2}m-n$.
\end{theo}
\begin{proof}
Using the same argument in the proof of Theorem~\ref{teo:particular}, the statement follows from the Smith normal form of the matrix $\M_{n,m}$, i.e. $\S_{n,m}=\diag(m,m(n+1),\ldots,m(n+1))=\U_n\M_{n,m}\V_n$. The unimodular matrices $\U_n$ and $\V_n$ are as in (\ref{eq:UiV}). The digraph isomorphism is now $D_{n,m}\cong G_{\M_{n,m}}$, where $G_{\M_{n,m}}$ is the commutative step digraph of Proposition~\ref{pro:general}.
\end{proof}

In terms of the degree $d=n$ and diameter $k$, we get the number of vertices and density of $D_{n,m}$
\begin{equation}
N_{d,k}=\frac{2^d}{d+1}\left( \frac{k}{d}+1\right)^d\quad\textrm{and}\quad \delta_{d}=\frac1{d+1}\left(\frac2d\right)^d\textrm{ for all }k.\label{eq:od}
\end{equation}
Note that, given a fixed degree $d$, the value of the density remains constant as the diameter $k$ increases.

%Dougherty and Faber (2004): A nonconstructive proof:
%
%There is a positive constant $c$ (not depending on $d$ or $k$), such that, for any fixed $d\ge 2$ and any $k$, there exist Cayley digraphs of Abelian groups on $g$ generators having diameter at most $k$ and number of vertices $N_{d,k}$ satisfying:
%$$
%\tb{N_{d,k}\ge \frac{c}{d(\ln d)^{1+\log_2 e}}\frac{k^d}{d!}+O(k^{d-1})}.
%$$

As mentioned before in the introduction,  Dougherty and Faber gave the following nonconstructive result in \cite{DoFa04}:
There is a positive constant $c$ (not depending on $d$ or $k$), such that, for any fixed $d\ge 2$ and any $k$, there exist Cayley digraphs of Abelian groups  having diameter at most $k$ and number of vertices satisfying:
$$
N_{d,k}\ge \frac{c}{d(\ln d)^{1+\log_2 e}}\frac{k^d}{d!}+O(k^{d-1}).
$$
which, using Stirling's formula, gives
\begin{equation}
\label{their-bound}
N_{d,k}\geq \frac{c}{\sqrt{2\pi}}e^{d-\frac{3}{2}\ln d-(\ln\ln d)(1+\log_2 e)}\left(\frac{k}{d} \right)^d+ O(k^{d-1}),
\end{equation}
with multiplicative factor of $\left(\frac{k}{d} \right)^d$ being
$$
\frac{c}{\sqrt{2\pi}}e^{d-\frac{3}{2}\ln d-(\ln\ln d)(1+\log_2 e)}\sim e^{d-\frac{3}{2}\ln d}.
$$
For large $k$, the order in (\ref{eq:od}) is
\begin{equation}
\label{our-bound}
N_{d,k}=2^{d-\log_2(d+1)}\left(\frac{k}{d} \right)^d+ O(k^{d-1}),
\end{equation}
where the multiplicative factor of $\left(\frac{k}{d} \right)^d$ is
$$
2^{d-\log_2(d+1)}\sim 2^{d-\frac{1}{\ln 2}\ln d}.
$$
Although this last coefficient turns out to be smaller than the one appearing in the theoretical bound of Dougherty and Faber, the explicit constructions given here achieve it and they are asymptotically the only ones known up to date.

%\begin{figure}[htb]
%\begin{center}
%\includegraphics*[scale=1]{graph-no-d-r}
%\end{center}
%\caption{Three drawings of a walk regular graph which is not
%distance-regular.} \protect\label{fig:graph no dist-reg}
%\end{figure}

\end{document}